\documentclass[11pt]{amsart}
\usepackage[utf8]{inputenc}
\usepackage{amsmath}
\usepackage{amsthm}
\usepackage{hyperref}
\usepackage{cleveref}
\usepackage{amssymb}
\usepackage[left=1.5in,right=1.5in]{geometry}
\usepackage{todonotes}
\usepackage{tikz}
\usepackage{tikz-cd}
\usepackage{mathtools}
\usepackage{enumerate}

\newtheorem{theorem}{Theorem}
\numberwithin{theorem}{section}
\newtheorem{lemma}[theorem]{Lemma}
\newtheorem{proposition}[theorem]{Proposition}
\newtheorem{conjecture}[theorem]{Conjecture}

\theoremstyle{definition}
\newtheorem{definition}[theorem]{Definition}

\theoremstyle{remark}
\newtheorem{remark}[theorem]{Remark}

\Crefname{conjecture}{Conjecture}{Conjectures}

\usetikzlibrary{decorations.markings, arrows.meta}
\tikzcdset{
 arrow style=tikz,
 diagrams={>={Straight Barb[scale=0.8]}},
 marrow/.style={decoration={markings,mark=at position 0.5 with {\arrow{#1}}}, postaction=decorate},
 rmarrow/.style={decoration={markings,mark=at position 0.8 with {\arrow{#1}}}, postaction=decorate}
}

\title[On combinatorial invariance of parabolic KL polynomials]{On combinatorial invariance of parabolic Kazhdan--Lusztig polynomials}
\author{Grant T. Barkley}
\author{Christian Gaetz}
\thanks{GTB was supported by NSF grant DMS-2152991.}
\address{Department of Mathematics, Harvard University, Cambridge, MA.}
\email{{\href{mailto:gbarkley@math.harvard.edu}{gbarkley@math.harvard.edu}}}
\address{Department of Mathematics, University of California, Berkeley, CA.}
\email{{\href{mailto:gaetz@berkeley.edu}{{\tt gaetz@berkeley.edu}}}}
\date{\today}

\begin{document}
\begin{abstract}
We show that the \emph{Combinatorial Invariance Conjecture} for Kazhdan--Lusztig polynomials due to Lusztig and to Dyer, its parabolic analog due to Marietti, and a refined parabolic version that we introduce, are equivalent. We use this to give a new proof of Marietti's conjecture in the case of lower Bruhat intervals and to prove several new cases of the parabolic conjectures.
\end{abstract}
\keywords{}
\subjclass{05E14, 14M15}
\maketitle
\section{Introduction}

The \emph{Kazhdan--Lusztig polynomial} $P_{u,v}(q) \in \mathbb{Z}[q]$, indexed by a pair $u,v$ of elements in a Coxeter group $W$, is a fundamental object in geometric representation theory \cite{Kazhdan-Lusztig-1}. These polynomials determine the transition from the standard basis of the Hecke algebra to the Kazhdan--Lusztig basis. When $W$ is the Weyl group of a complex semisimple Lie group $G$, the $P_{u,v}$ give both the Poincar\'{e} polynomials of the local intersection cohomology of Schubert varieties in $G/B$ and relate the characters of Verma modules and simple modules of $G$ \cite{Beilinson-Bernstein, Brylinski-Kashiwara}.

The polynomial $P_{u,v}$ is nonzero if and only if $u \leq v$ in \emph{Bruhat order}; the same is true of the \emph{Kazhdan--Lusztig $R$-polynomial} $R_{u,v}$, also introduced in \cite{Kazhdan-Lusztig-1}. The \emph{Combinatorial Invariance Conjecture} (\Cref{conj:cic}) asserts that, remarkably, both families of Kazhdan--Lusztig polynomials are completely determined by the combinatorics of Bruhat order. 

\begin{conjecture}[Lusztig c.~1983; Dyer \cite{Dyer1987}]
\label{conj:cic}
Suppose that for $u_1, v_1 \in W_1$ and $u_2,v_2 \in W_2$ the Bruhat intervals $[u_1,v_1]$ and $[u_2,v_2]$ are isomorphic as posets, then:
\begin{itemize}
    \item[(a)] $R_{u_1,v_1}=R_{u_2,v_2}$, and 
    \item[(b)] $P_{u_1,v_1}=P_{u_2,v_2}$.
\end{itemize}
\end{conjecture}

\Cref{conj:cic} has received considerable study, with many special cases having been established. In particular, it has been proven for \emph{lower intervals} \cite{Lower-intervals-general-type} and \emph{short edge intervals} \cite{Brenti-combinatorial}. It has also been verified for Coxeter groups of type $\widetilde{A}_2$ \cite{affine-A2} and for small rank finite Coxeter groups \cite{Incitti2007}. Work on this conjecture in the case of type $A$ has been especially active in recent years thanks to the conjectural approach of \cite{Blundell}.

\begin{remark}
    It is immediate from the definitions (see \Cref{def:parabolic-R,def:parabolic-P}) that the truth of \Cref{conj:cic}(a) for all pairs of intervals is equivalent to the truth of \Cref{conj:cic}(b) for all pairs of intervals.
\end{remark}

For $J$ a subset of the set $S$ of simple generators of $W$ and $x \in \{-1,q\}$, the \emph{parabolic Kazhdan--Lusztig polynomials} $P_{u,v}^{J,x}(q)$ generalize the $P_{u,v}$ and, in the case $x=-1$, give the Poincar\'{e} polynomials of the local intersection cohomology of Schubert varieties in $G/P_J$, where $P_J$ is the corresponding parabolic subgroup \cite{Deodhar-parabolic}. 
Here $u$ and $v$ lie in the parabolic quotient $W^J$. The \emph{parabolic $R$-polynomials} $R_{u,v}^{J,x}(q)$ likewise generalize the $R_{u,v}$. The $P_{u,v}$ and $R_{u,v}$ are the special case $J=\emptyset$.

Let $[u,v]^J \coloneqq [u,v] \cap W^J$ denote the set of elements from $[u,v]$ lying in $W^J$. Brenti, in the case of \emph{Hermitian symmetric pairs} \cite[Corollary 4.8]{Brenti-hermitian}, proved that parabolic Kazhdan--Lusztig polynomials are equal when $[u_1,v_1]^J \cong [u_2,v_2]^J$ (see also \cite{Boe,Lascoux-Schutzenberger}). But examples from \cite{Brenti-Mongelli-Sentinelli, Marietti-sm-for-parabolic} show that the na\"\i ve extensions of \Cref{conj:cic} that this suggests, where one only requires that $[u_1,v_1]^{J_1} \cong [u_2,v_2]^{J_2}$ or that $[u_1,v_1] \cong [u_2,v_2]$, are false in general. Thus any extension to the parabolic setting must include information about the isomorphism type of the Bruhat intervals as well as information about their intersections with $W^J$. Marietti proposed such an extension and proved it for lower intervals.

\begin{conjecture}[Marietti \cite{Marietti-sm-for-parabolic, Marietti-parabolic-cic}]
\label{conj:parabolic-cic}
Suppose that for $u_1, v_1 \in W_1^{J_1}$ and $u_2,v_2 \in W_2^{J_2}$ there is a poset isomorphism $\varphi: [u_1,v_1] \to [u_2,v_2]$ restricting to an isomorphism $[u_1,v_1]^{J_1} \to [u_2,v_2]^{J_2}$, then:
\begin{itemize}
    \item[(a)] $R^{J_1,x}_{u_1,v_1}=R^{J_2,x}_{u_2,v_2}$ for $x \in \{-1,q\}$, and
    \item[(b)] $P^{J_1,x}_{u_1,v_1}=P^{J_2,x}_{u_2,v_2}$ for $x \in \{-1,q\}$.
\end{itemize} 
\end{conjecture}

It was observed by Marietti \cite{Marietti-parabolic-cic} that the truth of \Cref{conj:parabolic-cic}(a) for all pairs of intervals is equivalent to the truth of \Cref{conj:parabolic-cic}(b) for all pairs of intervals.

In \Cref{conj:parabolic-cic-atoms} below, we propose a refined conjecture in which far less information about $[u,v]^J$ is required in order to determine $R^{J,x}_{u,v}$. Let 
\[
A^J_{u,v} \coloneqq \{a \in W^J \mid u \lessdot a \leq v\}
\]
denote the set of atoms of $[u,v]$ lying in $W^J$.

\begin{conjecture}
\label{conj:parabolic-cic-atoms}
Suppose that for $u_1, v_1 \in W_1^{J_1}$ and $u_2,v_2 \in W_2^{J_2}$ there is a poset isomorphism $\varphi: [u_1,v_1] \to [u_2,v_2]$ restricting to a bijection $A^{J_1}_{u_1,v_1} \to A^{J_2}_{u_2,v_2}$, then:
\begin{itemize}
    \item[(a)] $R^{J_1,x}_{u_1,v_1}=R^{J_2,x}_{u_2,v_2}$ for $x \in \{-1,q\}$, and 
    \item[(b)] $P^{J_1,q}_{u_1,v_1}=P^{J_2,q}_{u_2,v_2}$.
\end{itemize} 
\end{conjecture}

\begin{remark}
The analog of \Cref{conj:parabolic-cic-atoms} does not hold for the polynomials $\{P^{J,-1}_{u,v}\}$. For example, letting $W_1=W_2$ be the symmetric group $S_4$, $u_1=u_2=e$, $v_1=s_1s_2s_3$, and $v_2=s_2s_1s_3$, both $[u_1,v_1]$ and $[u_2,v_2]$ are isomorphic to the Boolean lattice $B_3$. Letting $J_1=\{s_1\}$ and $J_2=\{s_2\}$, we have $|A^{J_1}_{u_1,v_1}|=|A^{J_2}_{u_2,v_2}|=2$. Thus there is an isomorphism $\varphi$ satisfying the hypotheses of the conjecture. However $P^{J_1,-1}_{u_1,v_1}=1 \neq 1+q = P^{J_2,-1}_{u_2,v_2}$.
\end{remark} 

\Cref{conj:parabolic-cic-atoms}(a) implies \Cref{conj:parabolic-cic} since any poset isomorphism $[u_1,v_1]^{J_1}\to [u_2,v_2]^{J_2}$ in particular restricts to a bijection $A^{J_1}_{u_1,v_1}\to A^{J_2}_{u_2,v_2}$. \Cref{conj:parabolic-cic} in turn implies \Cref{conj:cic} by taking $J_1=J_2=\emptyset$. In our first main result, we show that the three conjectures are in fact equivalent. The equivalence of \Cref{conj:cic,conj:parabolic-cic} is already new.

\begin{theorem}
\label{thm:conj-are-equivalent}
\Cref{conj:cic,conj:parabolic-cic,conj:parabolic-cic-atoms} are equivalent.
\end{theorem}

In the case of \emph{lower intervals}, when $u_1$ and $u_2$ are the identity elements of $W_1$ and $W_2$ respectively, \Cref{conj:cic} was proven by Brenti--Caselli--Marietti \cite{Lower-intervals-general-type}. Later, Marietti proved \Cref{conj:parabolic-cic} for lower intervals \cite{Marietti-sm-for-parabolic,Marietti-parabolic-cic}. We prove \Cref{conj:parabolic-cic-atoms} for lower intervals and give a new short proof, relying on Brenti--Caselli--Marietti's results, of \Cref{conj:parabolic-cic} in this case. 

\begin{theorem}
\label{thm:lower-intervals}
\Cref{conj:cic,conj:parabolic-cic,conj:parabolic-cic-atoms} hold in the case $u_1=e_1$ and $u_2=e_2$ are the identity elements of $W_1$ and $W_2$.
\end{theorem}

We say a Bruhat interval $[u,v]$ is a \emph{short edge interval} if all edges $y \to y'$ in the Bruhat graph restricted to $[u,v]$ have $\ell(y')-\ell(y)=1$. \Cref{conj:cic} was proven for short edge intervals by Brenti \cite{Brenti-combinatorial}. Note that short edge intervals coincide with the $S_3$-free intervals studied there, by \cite[Proposition 3.3]{Dyer-bruhat-graph}.

\begin{theorem}
\label{thm:short-edge-intervals}
\Cref{conj:cic,conj:parabolic-cic,conj:parabolic-cic-atoms} hold when $[u_1,v_1]$ is a short edge interval.
\end{theorem}

Suppose that $W=S_n$ is the symmetric group. We call the Bruhat interval $[u,v] \subset W$ \emph{cosimple} if $\{\mathbf{e}_i-\mathbf{e}_j \mid i<j, u \leq v \cdot (ij) \lessdot v\}$ is linearly independent, where the $\mathbf{e}_i$ are the standard basis vectors in $\mathbb{R}^n$. We call $[u,v] \subset S_n$ \emph{coelementary} if it is isomorphic (as a poset) to some cosimple interval in some symmetric group. The coelementary intervals are related by poset duality to the \emph{elementary} intervals studied in \cite{elementary-paper}, but the coelementary convention will be more useful for our purposes here.

\begin{theorem} 
\label{thm:coelementary}
\Cref{conj:cic}(a), \Cref{conj:parabolic-cic}(a), and \Cref{conj:parabolic-cic-atoms}(a) hold when $W_1$ and $W_2$ are symmetric groups and $[u_1,v_1]$ is a coelementary interval.
\end{theorem}

\Cref{sec:prelim} gives background on Bruhat order and Kazhdan--Lusztig polynomials. In \Cref{sec:invariant-collections} we introduce \emph{invariant collections} of Bruhat intervals and prove \Cref{thm:invariant-I} which gives general criteria for the combinatorial invariance of parabolic Kazhdan--Lusztig polynomials. \Cref{thm:conj-are-equivalent,thm:lower-intervals,thm:short-edge-intervals,thm:coelementary} will be shown in \Cref{sec:proofs} to follow from \Cref{thm:invariant-I} and known special cases of \Cref{conj:cic}.

\section{Preliminaries}
\label{sec:prelim}
We refer the reader to \cite{Bjorner-Brenti} for background on Coxeter groups, Bruhat order, and Kazhdan--Lusztig polynomials.

\subsection{Parabolic decompositions and Bruhat order}

Throughout this work, $W$ will denote a Coxeter group with standard generating set $S$ and length function $\ell$, and $J$ a subset of $S$. We write $W_J$ for the parabolic subgroup of $W$ generated by $J$, and $W^J$ for the set of minimum-length representatives of the cosets $W/W_J$. Each element $w \in W$ may be uniquely decomposed as $w=w^Jw_J$ with $w^J \in W^J$ and $w_J \in W_J$. Furthermore, this decomposition satisfies $\ell(w)=\ell(w^J)+\ell(w_J)$.

We denote by $\leq$ the \emph{(strong) Bruhat order}, a partial order on $W$. The following standard fact will be useful (see, e.g. \cite[Prop.~2.5.1]{Bjorner-Brenti}).

\begin{proposition}
\label{prop:leq-on-quotient}
    For $u,v \in W$ and $J \subseteq S$, if $u \leq v$ then $u^J \leq v^J$.
\end{proposition}

We write $[u,v]$ for the closed interval $\{y \in W \mid u \leq y \leq v\}$. The \emph{Bruhat graph} is the directed graph with vertex set $W$ and directed edges $y \to y'$ whenever $y'=yt$ for some reflection $t$ and $\ell(y)<\ell(y')$.
\subsection{Kazhdan--Lusztig polynomials}

\begin{definition}[Kazhdan--Lusztig \cite{Kazhdan-Lusztig-1}; Deodhar \cite{Deodhar-parabolic}]
\label{def:parabolic-R}
Let $W$ be a Coxeter group, let $J \subset S$, and let $x \in \{-1,q\}$. The family of polynomials $\{R_{u,v}^{J,x}\}_{u,v \in W^J}$ is uniquely determined by the following conditions.
\begin{itemize}
    \item[(i)] $R_{u,v}^{J,x} = 0$ if $u \not \leq v$.
    \item[(ii)] $R_{u,u}^{J,x}=1$ for all $u \in W^J$.
    \item[(iii)] For all $s \in S$ with $\ell(sv)<\ell(v)$ we have:
    \begin{align*}
        R_{u,v}^{J,x} &= \begin{cases} R_{su,sv}^{J,x}, &  \text{if $\ell(su)<\ell(u)$} \\
        (q-1)R_{u,sv}^{J,x}+qR_{su,sv}^{J,x}, &  \text{if $\ell(su)>\ell(u)$ and $su \in W^J$} \\
        (q-1-x)R_{u,sv}^{J,x}, &  \text{if $\ell(su)>\ell(u)$ but $su \not \in W^J$.} 
        \end{cases}
    \end{align*}
\end{itemize}
\end{definition} 

\begin{definition}[Kazhdan--Lusztig \cite{Kazhdan-Lusztig-1}; Deodhar \cite{Deodhar-parabolic}]
\label{def:parabolic-P}
Let $W$ be a Coxeter group, let $J \subset S$, and let $x \in \{-1,q\}$. The family of polynomials $\{P_{u,v}^{J,x}\}_{u,v \in W^J}$ is uniquely determined by the following conditions.
\begin{itemize}
    \item[(i)] $P_{u,v}^{J,x} = 0$ if $u \not \leq v$.
    \item[(ii)] $P_{u,u}^{J,x}=1$ for all $u \in W^J$.
    \item[(iii)] $\deg P_{u,v}^{J,x} \leq \frac{1}{2} (\ell(v)-\ell(u)-1)$ if $u<v$.
    \item[(iv)] $q^{\ell(v)-\ell(u)}P^{J,x}_{u,v}(q^{-1}) = \sum_{\sigma \in [u,v]^J} R^{J,x}_{u,\sigma} P^{J,x}_{\sigma,v}$.
\end{itemize}
\end{definition}

When $J=\emptyset$, it is often omitted from the notation for the $R^J$ and $P^J$, and in this case the polynomials are independent of the choice of $x \in \{-1,q\}$; these are the \emph{ordinary} Kazhdan--Lusztig and $R$-polynomials. The following result of Deodhar allows parabolic Kazhdan--Lusztig polynomials to be expressed in terms of ordinary Kazhdan--Lusztig polynomials.

\begin{theorem}\cite[Prop.~2.12 \& Rem.~3.8]{Deodhar-parabolic}
\label{thm:deodhar-parabolic-as-sum}
    For $u,v \in W^J$ we have:
    \begin{itemize}
        \item[(a)] $R^{J,x}_{u,v} = \sum_{w \in W_J} (-x)^{\ell(w)} R_{uw,v}$, for $x \in \{-1,q\}$.
        \item[(b)] $P^{J,q}_{u,v} = \sum_{w \in W_J} (-1)^{\ell(w)} P_{uw,v}$.
    \end{itemize}
\end{theorem}

\section{Invariant collections of intervals}

We now define \emph{invariant collections} of Bruhat intervals. For such collections of intervals we will be able to transfer information about the combinatorial invariance of ordinary $P$- and $R$-polynomials to the parabolic setting (see \Cref{thm:invariant-I}).

\label{sec:invariant-collections}
\begin{definition}
Let $\mathcal{I}$ be a collection of Bruhat intervals in Coxeter groups. We say that $\mathcal{I}$ is:
\begin{itemize}
    \item \emph{upper $R$-invariant} (resp. \emph{upper $P$-invariant}) if for all $[u_1,v_1], [u_2,v_2] \in \mathcal{I}$ and for all poset isomorphisms $\varphi: [u_1,v_1] \to [u_2,v_2]$ we have $R_{y,v_1}=R_{\varphi(y),v_2}$ (resp. $P_{y,v_1}=P_{\varphi(y),v_2}$) for all $y \in [u_1,v_1]$.
    \item \emph{fully invariant} if for all $[u_1,v_1], [u_2,v_2] \in \mathcal{I}$ and for all poset isomorphisms $\varphi: [u_1,v_1] \to [u_2,v_2]$ we have $R_{y,y'}=R_{\varphi(y),\varphi(y')}$ for all $y,y' \in [u_1,v_1]$.
\end{itemize} 
\end{definition}

The following fact is immediate from \Cref{def:parabolic-P} (taking $J=\emptyset$ in the definition so that $W^J=W$).

\begin{proposition}
\label{prop:fully-R-implies-fully-P}
If a collection $\mathcal{I}$ of Bruhat intervals is fully invariant, then for all $[u_1,v_1], [u_2,v_2] \in \mathcal{I}$ and for all poset isomorphisms $\varphi: [u_1,v_1] \to [u_2,v_2]$ we have $P_{y,y'}=P_{\varphi(y),\varphi(y')}$ for all $y,y' \in [u_1,v_1]$.
\end{proposition}

\Cref{prop:fully-R-implies-fully-P} implies that fully invariant collections are in particular upper $R$-invariant and upper $P$-invariant.

\Cref{thm:invariant-I} below is our most general result. \Cref{thm:conj-are-equivalent,thm:lower-intervals,thm:short-edge-intervals,thm:coelementary} will all be shown to follow from it.

\begin{theorem}
\label{thm:invariant-I}
Let $\mathcal{I}$ be a collection of Bruhat intervals in Coxeter groups. Let $[u_1,v_1], [u_2,v_2] \in \mathcal{I}$ with $u_1,v_1 \in W_1^{J_1}$ and $u_2,v_2 \in W_2^{J_2}$ and let $\varphi: [u_1,v_1] \to [u_2,v_2]$ be a poset isomorphism restricting to a bijection $A^{J_1}_{u_1,v_1} \to A^{J_2}_{u_2,v_2}$.
\begin{itemize}
    \item[(a)] If $\mathcal{I}$ is upper $R$-invariant, then $R^{J_1,x}_{u_1,v_1}=R^{J_2,x}_{u_2,v_2}$ for $x \in \{-1,q\}$.
    \item[(b)] If $\mathcal{I}$ is upper $P$-invariant, then $P^{J_1,q}_{u_1,v_1}=P^{J_2,q}_{u_2,v_2}$.
\end{itemize}
If moreover $\varphi$ restricts to an isomorphism $[u_1,v_1]^{J_1} \to [u_2,v_2]^{J_2}$, then:
\begin{itemize} 
    \item[(c)] If $\mathcal{I}$ is fully invariant, we have  $P^{J_1,x}_{u_1,v_1}=P^{J_2,x}_{u_2,v_2}$ and $R^{J_1,x}_{u_1,v_1}=R^{J_2,x}_{u_2,v_2}$ for $x \in \{-1,q\}$.
\end{itemize}
\end{theorem}

We first prove a key lemma. We remark that this lemma is also a special case of \cite[Proposition 3.4]{SentinelliProjection}, using the fact that $w\mapsto w^J$ is a projection in the sense described there \cite{MariettiCosets}.

\begin{lemma}
    \label{lem:set-is-combinatorial}
    Let $u,v \in W^J$ with $u \leq v$. Then
    \begin{equation}
    \label{eq:set-is-combinatorial}
            uW_J \cap [e,v] = \{y \in [u,v] \mid y \not \geq a \text{ for all } a \in A^J_{u,v} \},
    \end{equation}
    where $e$ is the identity element of $W$.
\end{lemma}
\begin{proof}
Suppose $y = uw \in uW_J \cap [e,v]$. Since $u \in W^J$, we have $u \leq uw$, so $y \in [u,v]$. If we had $y \geq a$ for some $a \in A^J_{u,v}$, then we would have $y^J \geq a^J=a > u$ by \Cref{prop:leq-on-quotient}, but this contradicts the uniqueness of the parabolic decomposition $y=uw$. Thus $y$ belongs to the right hand side of (\ref{eq:set-is-combinatorial}).

Suppose now that $y \in [u,v]$ satisfies $y \not \geq a$ for all $a \in A^J_{u,v}$; we must show that $y \in uW_J$. By \Cref{prop:leq-on-quotient}, we have $y^J \in [u,v]^J$. Since $[u,v]^J$ is graded by length \cite[Cor.~2.5.6]{Bjorner-Brenti} and has unique minimal element $u$, if $y^J > u$ then $y^J \geq a$ for some $a \in A^J_{u,v}$. But this cannot be the case, since we would have $a \not \leq y \geq y^J \geq a$. Thus we must have $y^J=u$, and so $y=uy_J \in uW_J$.
\end{proof}

\begin{proof}[Proof of \Cref{thm:invariant-I}]
Let $[u_1,v_1], [u_2,v_2] \in \mathcal{I}$ with $u_1,v_1 \in W_1^{J_1}$ and $u_2,v_2 \in W_2^{J_2}$ and let $\varphi: [u_1,v_1] \to [u_2,v_2]$ be a poset isomorphism restricting to a bijection $A^{J_1}_{u_1,v_1}\to A^{J_2}_{u_2,v_2}$. The key observation is that the right-hand side of (\ref{eq:set-is-combinatorial}) from \Cref{lem:set-is-combinatorial} is preserved by $\varphi$ (a fact which is not clear a priori for the left-hand side). Indeed, we have
\begin{align}
    \nonumber
    &\phantom{{}=}\varphi\left(\{y_1 \in [u_1,v_1] \mid \forall a \in A^{J_1}_{u_1,v_1},~y_1 \not \geq a \}\right) \\ 
    \nonumber
    &= \{\varphi(y_1) \in [u_2,v_2] \mid \forall a \in \varphi(A^{J_1}_{u_1,v_1}),~\varphi(y_1) \not \geq a \} \\
    &= \{y_2 \in [u_2,v_2] \mid \forall a \in A^{J_2}_{u_2,v_2},~ y_2 \not \geq a \}.
    \label{eq:phi-applied-to-set}
\end{align}
Where in the first equality we have used that $\varphi$ is a poset isomorphism and in the second equality we have used the fact that $\varphi$ sends $A^{J_1}_{u_1,v_1}$ to $A^{J_2}_{u_2,v_2}$.

\begin{itemize}
    \item[(a)] Suppose that $\mathcal{I}$ is upper $R$-invariant. By \Cref{thm:deodhar-parabolic-as-sum}(a) we have
    \[
    R^{J_1,x}_{u_1,v_1} = \sum_{w \in W_{J_1}} (-x)^{\ell(w)} R_{u_1w,v_1}.
    \]
    Since $R_{u_1w,v_1} =0$ unless $u_1w \leq v_1$, this sum is the same as
    \[
    \sum_{y_1 \in u_1W_{J_1} \cap [e,v_1]} (-x)^{\ell(y_1)-\ell(u_1)} R_{y_1,v_1}=\sum_{\substack{y_1 \in [u_1,v_1] \\ \forall a \in A^{J_1}_{u_1,v_1},~y_1 \not \geq a}} (-x)^{\ell(y_1)-\ell(u_1)} R_{y_1,v_1}.
    \]
    We have $R_{y_1,v_1}=R_{\varphi(y_1),v_2}$ by upper $R$-invariance and hence by (\ref{eq:phi-applied-to-set}) we conclude
    \[
    R^{J_1,x}_{u_1,v_1} = \sum_{\substack{y_2 \in [u_2,v_2] \\ \forall a \in A^{J_2}_{u_2,v_2},~y_2 \not \geq a}} (-x)^{\ell(y_2)-\ell(u_2)} R_{y_2,v_2} = R^{J_2,x}_{u_2,v_2}.
    \]
    \item[(b)] If instead $\mathcal{I}$ is upper $P$-invariant, then we can apply \Cref{thm:deodhar-parabolic-as-sum}(b) and argue as above to conclude that $P^{J_1,q}_{u_1,v_1}=P^{J_2,q}_{u_2,v_2}$.
    \item[(c)] Suppose now that $\varphi$ restricts to an isomorphism $[u_1,v_1]^{J_1} \to [u_2,v_2]^{J_2}$ and that $\mathcal{I}$ is fully invariant. For any $u_1', v_1' \in [u_1,v_1]^{J_1}$ the restriction $\varphi|_{[u_1',v_1']}$ is a poset isomorphism onto its image $[u_2',v_2']$, where $u_2'\coloneqq \varphi(u_1'), v_2'\coloneqq \varphi(v_1')$. Furthermore, it sends $A^{J_1}_{u_1',v_1'}$ to $A^{J_2}_{u_2',v_2'}$. Thus, by the arguments in part (a), we have that
    \begin{equation}
    \label{eq:all-R-equal}
        R^{J_1,x}_{u_1',v_1'}=R^{J_2,x}_{u_2',v_2'}.
    \end{equation}
    If $u_1=v_1$ then $u_2=v_2$ and we have $P^{J_1,x}_{u_1,v_1}=P^{J_2,x}_{u_2,v_2}=1$. If $u_1 < v_1$ then
    \begin{align}
    \nonumber q^{\ell(v_1)-\ell(u_1)}P^{J_1,x}_{u_1,v_1}(q^{-1})-P^{J_1,x}_{u_1,v_1}(q) &= \sum_{\substack{\sigma_1 \in W_1^{J_1} \\ u_1 < \sigma_1 \leq v_1}} R^{J_1,x}_{u_1,\sigma_1}(q) P^{J_1,x}_{\sigma_1,v_1}(q) \\ \nonumber
    &=\sum_{\substack{\sigma_2 \in W_2^{J_2} \\ u_2 < \sigma_2 \leq v_2}} R^{J_2,x}_{u_2,\sigma_2}(q) P^{J_2,x}_{\sigma_2,v_2}(q) \\ \label{eq:P-plus-reverse}
    &=q^{\ell(v_2)-\ell(u_2)}P^{J_2,x}_{u_2,v_2}(q^{-1})-P^{J_2,x}_{u_2,v_2}(q).
    \end{align}
    Here we have used (\ref{eq:all-R-equal}) and have assumed by induction on the height of the intervals that 
    $P^{J_1,x}_{\sigma_1,v_1}=P^{J_2,x}_{\varphi(\sigma_1),v_2}$ for $u_1 < \sigma_1 \leq v_1$.  
    The fact that $\varphi$ is a poset isomorphism implies that $\ell(v_1)-\ell(u_1)=\ell(v_2)-\ell(u_2)$. Together with the degree bound from \Cref{def:parabolic-P}(iii), equation (\ref{eq:P-plus-reverse}) then implies that $P^{J_1,x}_{u_1,v_1}=P^{J_2,x}_{u_2,v_2}$. 
\end{itemize}
\end{proof}

\section{Applications of Theorem~\texorpdfstring{\ref{thm:invariant-I}}{\ref*{thm:invariant-I}}}
\label{sec:proofs}

The proofs of the remaining theorems follow from \Cref{thm:invariant-I}.

\begin{proof}[Proof of \Cref{thm:conj-are-equivalent}]
    It was noted in the introduction that \Cref{conj:parabolic-cic-atoms}(a) implies \Cref{conj:parabolic-cic} which in turn implies \Cref{conj:cic}. Thus it suffices to show that \Cref{conj:cic} implies \Cref{conj:parabolic-cic-atoms}.

    Suppose that \Cref{conj:cic} holds. This implies that for all $W_1$ and $W_2$ the collection $\mathcal{I}$ of all Bruhat intervals in $W_1$ and $W_2$ is fully invariant and therefore is, in particular, upper $R$-invariant and upper $P$-invariant. 
    By \Cref{thm:invariant-I}(a) and (b), we have \Cref{conj:parabolic-cic-atoms} (a) and (b) respectively. 
\end{proof}

\begin{proof}[Proof of \Cref{thm:lower-intervals}]
    It follows from \cite[Thm.~7.8]{Lower-intervals-general-type} that the collection of lower intervals in $W_1$ and $W_2$ is fully invariant. Applying \Cref{thm:invariant-I}(a) and (b) proves \Cref{conj:parabolic-cic-atoms} (a) and (b) for lower intervals, and applying \Cref{thm:invariant-I}(c) yields this case of \Cref{conj:parabolic-cic}.
\end{proof}

\begin{proof}[Proof of \Cref{thm:short-edge-intervals}]
    It follows from \cite[Thm.~6.3]{Brenti-combinatorial} that the collection $\mathcal{I}$ of short edge intervals in $W_1$ and $W_2$ is fully invariant. Applying \Cref{thm:invariant-I}(a) and (b) proves \Cref{conj:parabolic-cic-atoms} (a) and (b) for intervals from $\mathcal{I}$, and applying \Cref{thm:invariant-I}(c) yields this case of \Cref{conj:parabolic-cic}.
\end{proof}

\begin{proof}[Proof of \Cref{thm:coelementary}]
    It was shown in \cite[Thm.~1.6]{elementary-paper} that \Cref{conj:cic}(a) holds when $W_1$ and $W_2$ are symmetric groups and $[u_1,v_1]$ and $[u_2,v_2]$ are elementary intervals. Multiplication by the longest element $w_0$ of $S_n$ induces an antiautomorphism of Bruhat order and sends elementary intervals to coelementary intervals. Since we also have $R_{u,v}=R_{w_0v,w_0u}$ for all $u \leq v \in S_n$ \cite[Exer.~5.10(b)]{Bjorner-Brenti}, 
    the result of \cite{elementary-paper} also implies that \Cref{conj:cic}(a) holds for coelementary intervals in symmetric groups. Upper subintervals $[y,v]$ of coelementary intervals $[u,v]$ are easily seen to be coelementary themselves. Thus the collection $\mathcal{I}$ of coelementary intervals in symmetric groups is upper $R$-invariant. Applying \Cref{thm:invariant-I}(a) yields \Cref{conj:parabolic-cic}(a) and \Cref{conj:parabolic-cic-atoms}(a).
\end{proof}

\section*{Acknowledgements}

We are grateful to Mario Marietti and Paolo Sentinelli for interesting discussions during the Workshop ``Bruhat order: recent developments and open problems" held at the University of Bologna.

\bibliographystyle{plain}
\bibliography{arxiv-v3}
\end{document}